\documentclass[final]{siamltex}
\usepackage{enumerate}
\usepackage{tikz}
\usetikzlibrary{snakes}
\usetikzlibrary{arrows}
\def\G{{\mathcal{G}}}
\def\poly{\mathrm{poly}}

\begin{document}
\title{A decomposition based proof for fast mixing of a Markov chain over
balanced realizations of a joint degree matrix\thanks{All authors acknowledge
financial support from grant \#FA9550-12-1-0405 from the U.S. Air Force Office
of Scientific Research (AFOSR) and the Defense Advanced Research Projects
Agency (DARPA).}}

\author{P\'eter L. Erd\H os\footnotemark[2]\ \footnotemark[5]
\and Istv\'an Mikl\'os\footnotemark[2]\ \footnotemark[6] \footnotemark[7]
\and Zolt\'an Toroczkai\footnotemark[3] \footnotemark[4] \footnotemark[8] }

\renewcommand{\thefootnote}{\fnsymbol{footnote}}
\footnotetext[2]{Alfr\'ed R{\'e}nyi Institute, Re\'altanoda u 13-15 Budapest, 1053 Hungary.
{\tt email:} \{erdos.peter, miklos.istvan\}@renyi.mta.hu}
\footnotetext[3]{Department of Physics and Interdisciplinary Center for Network Science
\& Applications \\ University of Notre Dame, Notre Dame, IN, 46556, USA.
 {\tt email:} toro@nd.edu}
\footnotetext[4]{Max Planck Institute for the Physics of Complex Systems, 01187
Dresden, Germany.}
\footnotetext[5]{Supported in part by the Hungarian NSF, under contract NK 78439}
\footnotetext[6]{Supported in part by Hungarian NSF, under contract  PD84297}
\footnotetext[7]{Correspondence to: I. Mikl\'os}
\footnotetext[8]{Supported in part by the Defense Threat Reduction
Agency, \#HDTRA 1-09-1-0039.}
\renewcommand{\thefootnote}{\arabic{footnote}}

\maketitle

\begin{abstract}
A joint degree matrix (JDM) specifies the number of connections between nodes of given degrees in a graph, for all degree pairs and uniquely determines the degree sequence of the graph. We consider the space of all balanced realizations of an arbitrary JDM, realizations in which the links between any two, fixed-degree groups of nodes are placed as uniformly as possible. We prove that a swap Markov Chain Monte Carlo (MCMC) algorithm in the space of all balanced realizations of an {\em arbitrary} graphical JDM mixes rapidly, i.e., the relaxation time of the chain is bounded  from above by a polynomial in the number of nodes $n$.  To prove fast mixing, we first prove a general factorization theorem similar to the Martin-Randall method for disjoint decompositions (partitions). This theorem can be used to bound from below the spectral gap with the help of fast mixing subchains within every partition and a bound on an auxiliary Markov chain between the partitions. Our proof of the general factorization theorem is direct and uses conductance based methods (Cheeger inequality).
\end{abstract}



\pagestyle{myheadings}
\thispagestyle{plain}
\markboth{P.L. ERD\H{O}S, I. MIKL\'OS AND Z. TOROCZKAI}{FAST MIXING MARKOV CHAIN FOR JOINT DEGREE MATRICES}

\section{Introduction}\label{sec:intro}

The sampling of simple graphs on fixed number of nodes and with given degree
sequence is a well studied problem both by the statistics community (binary contingency tables  \cite{Bezakova}, \cite{R2}, \cite{R3}, \cite{R1}, \cite{R4}, \cite{DiaconisGanolli95},) and the computer science community. Lately the interest in this sampling problem has been widening with applications ranging from social sciences, physics, biology to engineering. The sampling approaches can be classified roughly into two types, one using Markov Chain Monte Carlo (MCMC)  algorithms \cite{R6}, \cite{R7}, \cite{R8}, \cite{R5}, \cite{R9}, \cite{R10}, \cite{rand} based on simple moves such as edge swaps and the other using direct  construction methods \cite{StarConstSampling,kiraly} and importance sampling  \cite{StarConstSampling,scdir}.

The MCMC approach, while conceptually simple, presents a notoriously difficult question, namely proving (or disproving) the fast mixing nature of a proposed Markov chain.  A Markov chain mixes fast if its mixing time (expressed via  the inverse of the spectral gap; for a summary of other measures see \cite{Ganapathy}) has a polynomial upper bound in the size of the problem, usually taken as the number of nodes $n$. In a seminal paper Kannan, Tetali and Vempala \cite{ktv} conjectured that the natural MCMC algorithm based on edge swaps mixes rapidly for arbitrary graphical degree sequences and they proved it for regular bipartite graphs. In 2007 Cooper, Dyer and Greenhill \cite{cdg} proved fast mixing for regular degree sequences of undirected graphs and subsequently Greenhill gave a proof for regular directed graphs \cite{g11}. Recently, Mikl\'os, Erd\H{o}s and Soukup \cite{mes} proved fast mixing for half-regular bipartite graphs (nodes have to have the same degree on only one side of the bipartition), using a modified version of Sinclair's multicommodity flow method \cite{sinc}.  In spite of these results on particular degree sequences for simple graphs, the proof for the general case, however, remains elusive. Recall that in a simple graph no multiple edges exist between any pairs of nodes, and there are no self-loops either.

To gain a better understanding of the mixing problem for arbitrary degree sequences of simple graphs we may follow a different approach. In particular, we may consider the sampling problem on a class of graphs obeying stronger constraints than the degree sequence, however, constraints that can be used to implicitly define an {\em arbitrary} degree sequence.

One such possibility is the joint degree matrix problem first introduced by Patrinos and Hakimi \cite{PH73} and subsequently studied by Amanatidis, Green and Mihail \cite{agm08}, Stanton and Pinar \cite{stanton} and then by Czabarka et al \cite{JDM}. A joint degree matrix (JDM) fixes the number of edges between nodes of given degrees, for all degree pairs. This constraint is stronger than just fixing the degree sequence, however, it uniquely defines the degree sequence (see the next section). Thus, we may study the MCMC sampling problem and the associated mixing time question on the space of all the graphical realizations of a given JDM. This means that we study the degree based sampling problem on a subspace of it, restricted by the degree-degree correlations imposed through the JDM. In a different context, Stanton and Pinar \cite{stanton} propose a restricted swap operation based Markov chain over the space of realizations of a JDM, which was recently shown to be irreducible by Czabarka et.al \cite{JDM}, a necessary condition of a MCMC sampling algorithm. Note that the JDM problem is interesting in its own right, with applications in social sciences (the degree assortativity problem) and data driven modeling of real-world networks \cite{stanton}.

However, proving that a suitable MCMC algorithm is fast mixing over the space of all
realizations of a JDM is also difficult. Here, instead, we prove fast mixing on a subspace of realizations of an arbitrary JDM, namely the space of balanced realizations. As shown in \cite{JDM} and \cite{stanton}, all graphical JDMs admit balanced realizations, in which the partial degrees of nodes from a given degree class towards another degree class are as uniformly distributed as possible. To prove fast mixing, we first introduce a disjoint partition of the space of balanced realizations. The elements in each partition can be expressed as the union of almost regular graphs and almost semi-regular bipartite graphs (see Def \ref{def:union}), while the structure of the partitioning can be also described through the collection of half-regular bipartite graphs (see Section \ref{sec:str} for definitions). We then
develop a modified version of the factorization theorem (\cite{MR06}, Thm 3.2) in the disjoint decomposition method first introduced by Martin and Randall \cite{MR06} to provide a lower bound for the spectral gap that is more readily computable. The factorization theorem of Martin and Randall for disjoint decompositions is based on a result by Caracciolo, Pelissetto and Sokal originally introduced in the framework of simulated tempering \cite{CPS}. The case of non-disjoint, large-overlap decompositions has been treated earlier by Madras and Randall \cite{MR02}. The Martin-Randall method for disjoint decompositions introduces a projection Markov chain between the partitions defined with the help of average probabilities. Unfortunately it is difficult to estimate/bound the spectral gap in this chain, and they resort
to a lower bound defined via a Metropolis-Hastings chain \cite{MR02}. Our main Theorem \ref{theo:general} also works with a lower bound, however, it is more general and we provide a direct and short proof for it using conductance based arguments.

The paper is organized as follows: after preliminaries related to JDMs we present the
main theorem (Thm 2.3) for fast mixing and our proof strategy. We then describe in
Section \ref{sec:str} the structure of the space of balanced realizations of a JDM and present a disjoint partitioning of this space into  subspaces whose elements are formed by almost regular and almost semi-regular subgraphs.  In Section \ref{sec:factor} using conductance based arguments, we prove our general theorem for fast mixing over state spaces that can be disjointly partitioned as described in Section \ref{sec:str} . In Section \ref{sec:RSO} we briefly recall earlier results on fast mixing over realizations of regular degree sequences and half-regular bipartite degree sequences, then extend their proofs (in the Appendix) to almost regular and almost half-regular cases. We then introduce our Markov chain over the space of balanced realizations, and complete the proof of our main theorem.

\section{Preliminaries}\label{sec:pre}

A symmetric matrix $J$ with non-negative integer elements  is the {\em joint degree matrix} (JDM) of an undirected simple graph $G$ iff the element  $J_{i,j}$ gives the number of edges between the class $V_i$ of vertices all having degree $i > 0$ and the class $V_j$ of vertices all with degree $j>0$ in the graph.  In this case we also say that $J$ is {\em graphical} and  that $G$ is a {\em graphical realization} of $J$. Note that there can be many different graphical realizations of the same JDM.

Given a JDM, the number of vertices $n_i = |V_i|$ in class $i$ is obtained from:
\begin{equation}
n_i = \frac{ J_{i,i}+\sum_{j=1}^k J_{i,j} }{i}\;, \label{ni}
\end{equation}
where $k$ denotes the number of distinct degree classes. This implies that a JDM also
uniquely determines the degree sequence, since we have obtained the number of nodes of given degrees for all possible degrees. Clearly $k \leq \Delta$, where $\Delta$ is the maximum degree. For sake of uniformity we consider all vertex classes $V_k$ for $k=1,\ldots,\Delta$; therefore we consider empty classes with $n_k=0$ vertices as well.  A necessary condition for $J$ to be graphical is that all the $n_i$-s are integers. Let $n$ denote the total number of vertices. Naturally, $n = \sum_{i} n_i$ and it is uniquely determined via Eq (\ref{ni}) for a given graphical JDM. The necessary and sufficient conditions for a given JDM to be graphical are provided, for e.g., in paper \cite{JDM}.

Let $d_j(v)$ denote the number of  edges such that one end-vertex is $v$ and the other end-vertex belongs to $V_j$, i.e., $d_j(v)$ is the degree of $v$ in $V_j$. The vector consisting of the $d_j(v)$-s for all $j$ is called the {\em degree spectrum} of vertex $v$. We introduce the notation
$$
\Theta_{i,j}=\left\{
               \begin{array}{ll}
                 0, & \hbox{if } n_i=0\,, \\
                 \frac{J_{i,j}}{n_i}, & \hbox{otherwise,}
               \end{array}
             \right.
$$
which gives the average number of the neighbors of a degree-$i$ vertex in vertex class $V_j$. Then a realization of the JDM is {\em balanced} iff for every $i$ and all $v \in V_i$ and all $j$, we have
$$
\left|d_j(v) - \Theta_{i,j}\right| < 1\;.
$$

The following theorem is proven in paper \cite{JDM} as Corollary 5:
\begin{theorem}
\label{theo:existance}
Every graphical JDM admits a balanced realization.
\end{theorem}

A restricted swap operation (RSO) takes two edges $(x,y)$ and $(u,v)$ with $x$ and $u$
from the same vertex class and swaps them with two non-edges $(x,v)$ and  $(u,y)$. The
RSO preserves the JDM, and in fact forms an irreducible Markov chain over all   its
realizations \cite{JDM}. An RSO Markov chain restricted to {\em balanced realizations}  can be defined as follows:
\begin{definition}
Let $J$ be a JDM. The state space of the RSO Markov chain consists of all the balanced realizations of $J$. It was proved by Czabarka et al. {\rm \cite{JDM}} that this state space is connected under restricted swap operations. The transitions of the Markov chain are defined in the following way. With probability $1/2$, the chain does nothing, so it remains in the current state (we consider a lazy Markov chain). With probability $1/2$ the chain will chose four, pairwise disjoint vertices, $v_1, v_2, v_3,v_4$ from the current realization (the possible choices are order dependent) and check whether $v_1$ and $v_2$ are chosen from the same vertex class, furthermore whether the
$$
 E\setminus\{(v_1,v_3),(v_2,v_4)\} \cup   \{(v_1,v_4), (v_2,v_3)\}
$$
swap operation is feasible. If this is the case then our Markov chain performs the swap operation if it leads to another balanced JDM realization. Otherwise the Markov chain remains in the same state. {\rm (Note that exactly two different orders of the selected vertices will provide the same swap operation, since the roles of $v_1$ and $v_2$ are symmetric.)} Then there is a transition with probability
$$
\frac{1}{n(n-1)(n-2)(n-3)}
$$
between two realizations iff there is a RSO transforming one into the other.
\end{definition}

In this paper, we prove that such a Markov chain is rapidly mixing. The  convergence
of a Markov chain is measured as a function of the  input data size. Here we note that
the size of the data is the  number of vertices (or number of edges, they are polynomially
bounded functions of each other) and not the number of digits to describe the  JDM.
This distinction is important as, for example, one can create a $2 \times  2$ JDM with values
$J_{2,2}=J_{3,3} = 0$ and $J_{2,3}=J_{3,2} = 6n$,  which has  $\Omega(n)$ number
of vertices (or edges) but it needs only  $O(\log(n))$ number of digits to describe
(except in the unary number system). Alternatively, one might consider the input is given in unary.

Formally, we state the rapid mixing property via the following theorem:

\begin{theorem}
\label{theo:main}
The RSO Markov chain on balanced JDM realizations is a rapidly mixing Markov chain,
namely, for the second largest eigenvalue $\lambda_2$ of this chain, it holds that
$$
\frac{1}{1-\lambda_2} = O(\poly(n))
$$
where $n$ is the number of vertices in the realizations of the JDM.
\end{theorem}
Note that the expression on the LHS is called, with some abuse of notation, the {\em relaxation time}: it is the time is needed for the Markov chain to reach its stationary distribution. The proof is based on the special structure  of the state space of the balanced JDM realizations. This special structure allows the following proof strategy: if we can prove that some auxiliary Markov chains are rapidly mixing on some sub-spaces obtained from decomposing the above-mentioned specially structured state space, then the Markov chain on the whole space is also rapidly mixing. We are going to prove the rapid mixing of these auxiliary Markov chains, as well as give the proof of the general theorem, that a Markov chain on this special structure is rapidly mixing, hence proving our main Theorem~\ref{theo:main}.

\section{The structure of the space of balanced JDM realizations, and the
Markov chain over this space}\label{sec:str}

In order to describe the structure of the space of balanced JDM realizations, we first define the
almost semi-regular bipartite and almost regular graphs.
\begin{definition}
A bipartite graph $G(U,V;E)$ is \emph{almost semi-regular} if for any $u_1, u_2 \in U$ and $v_1,v_2 \in V$
$$\left|d(u_1) - d(u_2) \right| \le 1$$
and
$$\left|d(v_1) - d(v_2) \right| \le 1.$$
\end{definition}
\begin{definition}\label{def:almost}
A graph $G(V,E)$ is \emph{almost regular}, if for any $v_1,v_2 \in V$
\end{definition}
$$\left|d(v_1) - d(v_2) \right| \le 1.$$
It is easy to see that the restriction of any graphical realization of the JDM to vertex classes $V_i, V_j, i\ne j$ can be considered as the coexistence of two almost regular graphs (one on $V_i$ and the other on $V_j$), and one almost semi-regular bipartite graph on the vertex class pair $V_i,V_j$. More generally, the collection of these almost semi-regular bipartite graphs and almost regular graphs completely determines the balanced JDM realization. Formally:
\begin{definition}[Labeled union]\label{def:union}
Any balanced JDM realization can be represented as a set of almost semi-regular bipartite graphs and almost regular graphs. The realization then can be constructed from these factor graphs as their {\em labeled union}: the vertices with the same labels are collapsed, and the edge set of the union is the union of the edge sets of the factor graphs.
\end{definition}

It is useful to construct the following auxiliary graphs. For each vertex class $V_i$, we create an auxiliary bipartite graph, $\G_i(V_i,U;E)$, where $U$ is a set of ``super-nodes'' representing all vertex classes $V_j$, including $V_i$. There is an edge between $v \in V_i$ and super-node $u_j$ representing vertex class $V_j$ iff
$$
    d_j(v) = \left\lceil\Theta_{i,j}\right\rceil\;,
$$
i.e., iff node $v$ carries the ceiling of the average degree of its class $i$ towards the other class $j$. (For sake of uniformity we construct these auxiliary graphs for all $i=1,\ldots,\Delta$, even, if some of them have no edge at all. Similarly, all super-nodes are given, even if some of them has no incident edge.) We claim that these $\Delta$ auxiliary graphs are \emph{half-regular}, i.e., each vertex in $V_i$ has the same degree (the degrees in the vertex class $U$ might be arbitrary). Indeed, the vertices in $V_i$ all have the same degree in the JDM realization, therefore, the number of times they have the ceiling of the average degree towards a vertex class is constant  in a balanced realization.

Let $Y$ denote the space of all balanced realizations of a JDM and just as before, let $\Delta$ denote the number of the vertex classes (some of them can be empty). We will represent the elements of $Y$ via a vector $y$  whose $\Delta(\Delta+1)/2$ components are the $\Delta$ almost regular graphs and the $\Delta(\Delta-1)/2$ almost regular bipartite graphs from their labeled union decomposition, as described in definition \ref{def:union} above. Given an element $y \in Y$ (i.e., a balanced graphical realization of the JDM) it has $\Delta$ associated auxiliary graphs ${\cal G}_i(V_i,U;E)$, one for every vertex class $V_i$ (some of them can be empty graphs). We will consider this collection of auxiliary graphs for a given $y$ as a $\Delta$-dimensional vector $x$, where $x = ({\cal G}_1,\ldots,{\cal G}_{\Delta})$.

For any given $y$ we can determine the corresponding $x$ (so no particular $y$ can correspond to two different $x$s), however, for a given $x$ there can be several $y$-s with that same $x$.  We will denote by $Y_x$ the subset of $Y$ containing all  the $y$-s with the same (given) $x$ and by $X$ the set of all possible induced $x$ vectors. Clearly, the $x$ vectors can be used to define a disjoint partition on $Y$:  $Y = \bigcup\limits_{x \in X} Y_x$. For notational convenience we will consider the space $Y$ as pairs $(x,y)$,  indicating the $x$-partition to which $y$ belongs.
This should not be confused with the notation for an edge, however, this should be evident from the context.  A restricted swap operation might fix $x$ in which case it will make a move only within $Y_x$, but if it does not fix $x$ then it will change both $x$ and $y$. For any $x$, the RSOs moving only within $Y_x$ form a Markov chain. On the other hand, tracing only the $x$s from the pairs $(x,y)$ is not a Markov chain: the probability that an RSO changes $x$ (and thus also $y$) depends also on the current $y$ not only  on $x$. However, the following theorem holds:
\begin{theorem} \label{theo:swapexist}
Let $(x_1,y_1)$ be a balanced realization of a JDM in the above mentioned  representation.
\begin{enumerate}[{\rm (i)}]
\item Assume that $(x_2,y_2)$ balanced realization is derived from the first one with one restricted swap operation. Then either $x_1 = x_2$ or they differ in exactly one coordinate, and the two corresponding auxiliary graphs differ only in one swap operation.
\item Let $x_2$ be a vector differing only in one coordinate from $x_1$, and furthermore, only in one swap within this coordinate, namely, one swap within one coordinate is sufficient to transform $x_1$ into $x_2$. Then there exists at least one $y_2$ such that $(x_2,y_2)$ is a balanced JDM realization and $(x_1,y_1)$ can be transformed into $(x_2,y_2)$ with a single RSO.
\end{enumerate}

\end{theorem}
\begin{proof} (i) This is just the reformulation of the definitions  for the $(x,y)$ pairs. \\
(ii) (See also Fig.~\ref{fig:auxiliary_bipartite}) By definition there is a degree $i, 1 \le i \le \Delta$ such that auxiliary graphs $x_1({\cal G}_i)$ and  $x_2({\cal G}_i)$ are different and one swap operation transforms the first one into the second one. More precisely there are vertices $v_1,v_2 \in V_i$ such that the swap transforming $x_1({\cal G}_i)$ into $x_2({\cal G}_i)$ removes edges $(v_1,U_j)$ and $(v_2,U_k)$ (with $j \neq k$)  and adds edges $(v_1,U_k)$ and $(v_2,U_j)$. (The capital letters show that the second vertices are super-vertices.) Since the edge $(v_1,U_j)$ exists in the graph $x_1(G_1)$ and $(v_2,U_j)$ does not belong to graph $x_1(G_i)$, therefore $d_j(v_1) > d_j(v_2)$ in the realization $(x_1,y_1).$ This means that there is at least one vertex $w \in V_j$ such that $w$ is connected to $v_1$ but not to $v_2$ in the realization $(x_1,y_1)$. Similarly, there is at least one vertex $r \in V_k$ such that $r$ is connected to $v_2$ but not to $v_1$ (again, in realization $(x_1,y_1)$). Therefore, we have a required RSO on nodes $v_1,v_2,w,r$. \qquad\end{proof}
\begin{figure}[h!]
\begin{center}
\begin{tikzpicture}[auto,scale=.32]
\draw [fill=cyan!15] (3,0) ellipse (5cm and 1.5cm);
\draw [fill=pink!20] (-1,6) ellipse (5cm and 1.5cm);
\draw [fill=pink!20]  (10,12) ellipse (5cm and 1.5cm);
\draw [fill=yellow!20!white,rotate=15]  (24,2) ellipse (5cm and 2cm);
\draw [fill=cyan!15] (25,0) ellipse (5cm and 1.5cm);

\draw node at (-3.5,0) {\Large V$_i$};
\draw [very thick,red] (0,0) -- (-5,6);
\draw [very thick,red] (3,0) -- (0,6);
\draw [very thick,gray!40] (0,0) -- (-4,6);
\draw [very thick,gray!40] (3,0) -- (-1.5,6);
\draw [very thick,gray!40] (6,0) -- (2,6);
\draw [very thick,gray!40] (0,0) -- (7,12);
\draw [very thick,gray!40] (3,0) -- (9.5,12);
\draw [very thick,gray!40] (6,0) -- (11.5,12);
\draw [very thick,dashed,gray] (6,0) -- (0,6);
\draw [very thick,dashed,gray] (3,0) -- (13,12);
\draw [very thick,red] (6,0) -- (13,12);
\draw [fill=black] (0,0) circle (15pt);
\draw [fill=black] (3,0) circle (15pt);
\draw [fill=black] (6,0) circle (15pt);
\draw [fill=black] (1,0) circle (3pt);
\draw [fill=black] (2,0) circle (3pt);
\draw [fill=black] (4,0) circle (3pt);
\draw [fill=black] (5,0) circle (3pt);

\draw [fill=black] (13,12) circle (15pt);
\draw [fill=black] (0,6) circle (15pt);
\draw node at (3,-1) {$v_1$};
\draw node at (5.4,-0.8) {$v_2$};
\draw node at (25,-1) {$v_1$};
\draw node at (27.4,-0.82) {$v_2$};
\draw node at (-.5,7) {$w$};
\draw node at (14,12) {$r$};

\draw [very thick,red] (22,0) -- (20,7.5);
\draw [very thick,red] (25,0) -- (20,7.5);
\draw [very thick,red] (28,0) -- (26,9);
\draw [very thick,dashed,gray] (25,0) -- (26,9);
\draw [very thick,dashed,gray] (28,0) -- (20,7.5);
\draw [fill=black] (22,0) circle (15pt);
\draw [fill=black] (25,0) circle (15pt);
\draw [fill=black] (28,0) circle (15pt);
\draw [fill=black] (23,0) circle (3pt);
\draw [fill=black] (24,0) circle (3pt);
\draw [fill=black] (26,0) circle (3pt);
\draw [fill=black] (27,0) circle (3pt);

\draw node at (-7.5,6) {\Large V$_j$};
\draw node at (3.5,12) {\Large V$_k$};
\draw node at (18.5,0) {\Large V$_i$};
\draw node at (18.8,7.4) {$U_j$};
\draw node at (24.6,9.1) {$U_k$};
\draw [fill=black] (20,7.5) circle (15pt);
\draw [fill=black] (26,9) circle (15pt);
\draw [fill=black] (2.5,8) circle (3pt);
\draw [fill=black] (3,9) circle (3pt);
\draw [fill=black] (3.5,10) circle (3pt);
\draw [fill=black] (22,7.8) circle (3pt);
\draw [fill=black] (23, 8.1) circle (3pt);
\draw [fill=black] (24,8.4) circle (3pt);
\draw [very thick,dashed,->] (15.5,12) .. controls +(2,2) and +(-2,2) .. (25.5,9.5);
\draw [very thick,dashed,->] (4.5,6.2) .. controls +(2,2) and +(-3,-3) .. (19.3,6.8);
\draw [line width=3mm,->, gray] (12,2) -- (16,2);
\end{tikzpicture}
\end{center}
\caption{Construction of the auxiliary bipartite graph $\G_i$ and a RSO
$\{(v_1,w),(v_2,r)\} \mapsto \{(v_1,r),(v_2,w)\}$ taking
$(x_1,y_1)$ into $(x_2,y_2)$.}\label{fig:auxiliary_bipartite}
\end{figure}
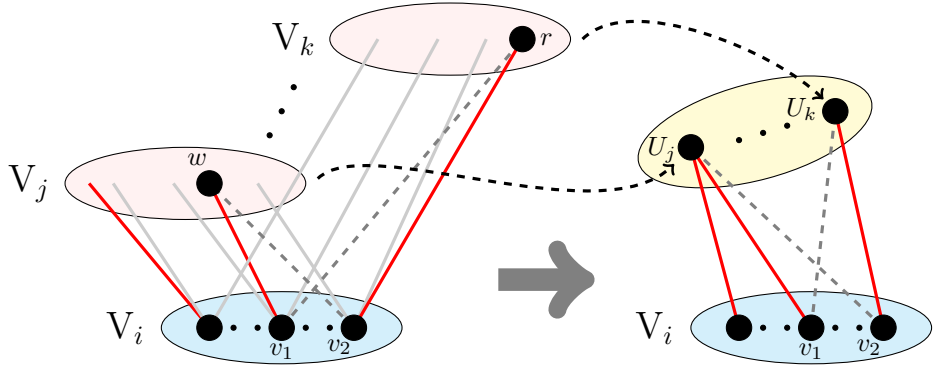

Thus any RSO on a balanced realization yielding another balanced realization either does not change $x$ or changes $x$ exactly on one coordinate (one auxiliary graph), and this change can be described with a swap,  taking one auxiliary graph into the other.

\section{Proving rapid mixing of Markov chains over factorized state spaces}\label{sec:factor}

In this section we will prove a general factorization theorem (Thm~\ref{theo:general}) on
fast convergence of a Markov chain. The proof of this theorem will lead to the proof of our
main result, Theorem \ref{theo:main}. Our theorem is similar to a theorem by Martin-Randall
and the comparison between the theorems is given at the end of this section. In the following
we will need the Cheeger inequality and a slight modification of it.

Let $M$ be a discrete time, discrete space, reversible Markov chain over set $\mathcal{I}$
with stationary distribution $\pi$ and transition probabilities from
$a$ to $b$ denoted by $T(b|a)$, where $a,b \in \mathcal{I}$.
The probability of a subset is denoted by
$$
\pi(S) := \sum_{a \in S} \pi(a)
$$
The conditional flow out of a subset of the state space $S \subset \mathcal{I}$ is defined by
\begin{equation}\label{eq:cond}
\Psi(S) := \frac{\sum\limits_{a \in S, b \in \bar{S}} \pi(a) T(b|a)} {\pi(S)}\;,
\end{equation}
where $\bar{S}$ denotes the complementary set of $S$ in $\mathcal{I}$. The conductance
of the state space is defined as
$$
\Phi := \min_{S}\left\{\Psi(S)\quad \middle| \quad S \subset \mathcal{I} \mbox{ and } 0 < \pi(S) \le \frac{1}{2}\right\}.
$$
The Cheeger inequality quoted in Theorem \ref{CIq} gives lower and upper bounds on
the second largest eigenvalue of the Markov chain, see for example the paper of Diaconis and Stroock (\cite{ds1991}) for a proof.
\begin{theorem}[Cheeger inequality]\label{theo:cheeger}
\label{CIq}
$$
1 - 2\Phi \le \lambda_2 \le 1 - \frac{\Phi^2}{2}\;.
$$
\end{theorem}

Here we prove a variant of the left Cheeger inequality (the lower-bound).

\begin{lemma}
\label{lemma:mod}
For any reversible Markov chain, and any subset $S$ of its state space,
\begin{equation}\label{eq:mod}
\frac{1-\lambda_2}{2} \min\{\pi(S),\pi(\bar{S})\} \le \sum_{a \in S, b \in \bar{S}} \pi(a) T(b|a)\;.
\end{equation}
\end{lemma}

\begin{proof}
The right hand side of Equation~\ref{eq:mod} is symmetric due to the reversibility of the chain. Thus, if $\pi(S) > \frac{1}{2}$, then $S$ and $\bar{S}$ can be switched. If $\pi(S) \le \frac{1}{2}$, the inequality is simply a rearrangement of the Cheeger inequality (the left inequality in Theorem~\ref{CIq}. \quad
\end{proof}

Now we are ready to state and prove a general theorem on rapidly
mixing Markov chains.

\begin{theorem} \label{theo:general}
Let ${\cal M}$ be a class of irreducible, aperiodic, reversible Markov chains whose state space $Y$ can be partitioned into disjoint classes $Y = \cup_{x \in X} Y_x$ by the elements of some set $X$. For notational convenience we also denote the element $y \in Y_x$ via the pair $(x,y)$ to indicate the partition it belongs to. The problem size of a particular chain is denoted by $n$. Let $T$ be the transition matrix of $M \in {\cal M}$, and let $\pi$ denote the stationary distribution of $M.$ Moreover,  let $\pi_X$ denote the marginal of $\pi$ on the first coordinate that is,  $\pi_X(x)= \pi(Y_x)$ for all $x$. Also, for arbitrary but fixed $x$ let us denote by $\pi_{Y_x}$ the stationary probability distribution restricted to  $Y_x$, i.e., $\pi(y)/\pi(Y_x)$, $\forall y \in Y_x$.  Assume that the following properties hold:
\begin{enumerate}[{\rm (i)}]
\item For all $x$, the transitions  with $x$ fixed form an aperiodic, irreducible and reversible Markov chain denoted by $M_x$ with stationary distribution $\pi_{Y_x}$. This Markov chain $M_x$ has transition probabilities as Markov chain $M$ for all transitions fixing $x$, except loops, which have increased probabilities such that the transition probabilities sum up to $1$. All transitions that would change $x$ have $0$ probabilities. Furthermore, this Markov chain is rapidly mixing, i.e., for its second largest eigenvalue $\lambda_{M_x,2}$ it holds  that
    $$\frac{1}{1-\lambda_{M_x,2}} \le \poly_1(n). $$
\item There exists a Markov chain $M'$ with state space $X$ and with transition matrix $T'$   which is aperiodic, irreducible and reversible w.r.t. $\pi_X$, and   for all $x_1,y_1,x_2$ it holds that
\begin{equation}
\sum_{y_2 \in Y_{ x_{2}}} T((x_2,y_2)|(x_1,y_1)) \ge T'(x_2|x_1). \label{eq:condition}
\end{equation}
Furthermore, this Markov chain is rapidly mixing, namely, for its second largest eigenvalue $\lambda_{M',2}$ it holds that
$$
\frac{1}{1-\lambda_{M',2}} \le \poly_2(n).
$$
\end{enumerate}
Then $M$ is also rapidly mixing as its second largest eigenvalue obeys:
$$
\frac{1}{1-\lambda_{M,2}} \le \frac{256 \poly_1^2(n)  \poly_2^2(n)} {\left(1-\frac{1}{\sqrt{2}}\right)^4}
$$
\end{theorem}

\begin{proof}
For any non-empty subset $S$ of the state space $Y=\bigcup\limits_x Y_x$ of $M$ we define
$$
X(S) := \{x \in X \ |\ \exists y, (x,y)\in S\}
$$
and for any given $x\in X$ we have
$$
Y_x(S) := \{(x,y)\in Y \ | \  (x,y) \in S\} = Y_x \cap S.
$$
We are going to prove that the conditional flow $\Psi (S)$ (see equation (\ref{eq:cond})) from any $S\subset Y$ with $0 < \pi(S) \le 1/2$ cannot be too small and therefore, neither the conductance of the Markov chain will be small. We cut the state space into two parts $Y= Y^l \cup Y^u$, namely the lower and upper parts using the following definitions (see also Fig.~\ref{fig:slsu_def}): the partition $X= L \cup U$ is defined as
\begin{eqnarray*}
L &:=& \left\{x \in X \, \middle| \frac{\pi(Y_x(S))}{\pi(Y_x)} \le 1/\sqrt{2} \right\}, \\
U &:=& \left\{x \in X \,  \middle| \frac{\pi(Y_x(S))}{\pi(Y_x)} > 1/\sqrt{2} \right\}\,.
\end{eqnarray*}
Furthermore, we introduce:
$$
Y^l := \bigcup_{x\in L} Y_x \quad \mbox{and}\quad Y^u := \bigcup_{x\in U} Y_x\,,
$$
and finally let
$$
S_l := S \cap Y^l \quad \mbox{and} \quad S_u := S \cap Y^u.
$$

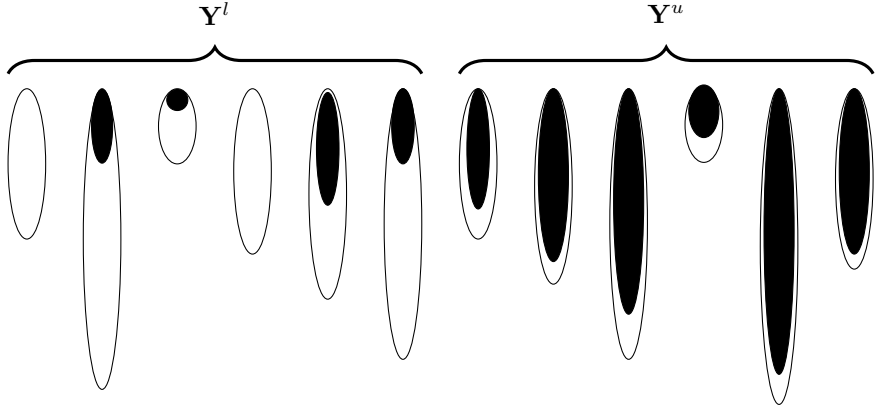
\begin{figure}[h]
\begin{center}
\begin{tikzpicture}
\pgfsnakesegmentamplitude=10pt
\draw [snake=brace, very thick] (0,5.2) -- (5.5,5.2);
\node at (2.75, 6) {$\mathbf{Y}^l$};
\draw (.25,4) ellipse (.25 cm and 1cm);
\draw (1.25,3) ellipse (.25 cm and 2cm);
\draw [fill=black] (1.25,4.5) ellipse (.14 cm and .49cm);
\draw (2.25,4.5) ellipse (.25 cm and .5cm);
\draw [fill=black] (2.25,4.85) circle (.14 cm);
\draw (3.25,3.9) ellipse (.25 cm and 1.1cm);
\draw (4.25,3.6) ellipse (.25 cm and 1.4cm);
\draw [fill=black] (4.25,4.2) ellipse (.15 cm and .75cm);
\draw (5.25,3.2) ellipse (.25 cm and 1.8cm);
\draw [fill=black] (5.25,4.5) ellipse (.15 cm and .5cm);
\draw [snake=brace, very thick] (6,5.2) -- (11.5,5.2);
\node at (8.75, 6) {$\mathbf{Y}^u$};
\draw (6.25,4) ellipse (.25 cm and 1cm);
\draw [fill=black] (6.25,4.2) ellipse (.15 cm and .8cm);
\draw (7.25,3.7) ellipse (.25 cm and 1.3cm);
\draw [fill=black] (7.25,3.85) ellipse (.2 cm and 1.15cm);
\draw (8.25,3.2) ellipse (.25 cm and 1.8cm);
\draw [fill=black] (8.25,3.5) ellipse (.2 cm and 1.5cm);
\draw (9.25,4.52) ellipse (.25 cm and .5cm);
\draw [fill=black] (9.25,4.7) ellipse (.2 cm and .35cm);
\draw (10.25,2.9) ellipse (.25 cm and 2.1cm);
\draw [fill=black] (10.25,3.1) ellipse (.2 cm and 1.9cm);
\draw (11.25,3.8) ellipse (.25 cm and 1.2cm);
\draw [fill=black] (11.25,3.9) ellipse (.2 cm and 1.1cm);
\end{tikzpicture}
\end{center}
\caption{The structure of $Y=Y^l \cup Y^u$.
A non-filled ellipse (with a simple line boundary) represents the
space $Y_x$ for a given $x$. The solid black ellipses represent the set $S$ with
some of them
(the $S_l$) belonging to the lower part $Y^l$, and the rest (the $S_u$)
belonging to the upper part ($Y^u$).}\label{fig:slsu_def}
\end{figure}

\noindent Since $M'$ is rapidly mixing we can write (based on Theorem \ref{CIq}):
$$
1-2\Phi_{M'} \le \lambda_{M',2} \le 1- \frac{1}{\poly_2(n)}\;,
$$
or
$$
\Phi_{M'} \ge \frac{1}{2\poly_2(n)}\;.
$$
We use this lower bound of conductance to define two cases regarding
the lower and upper part of $S$. Without loss of generality, we can
assume that $\poly_2(n) > 1$ for all positive $n$, a condition what we
need later on for technical reasons.
\begin{enumerate}[{\bf Case} 1]
\item We say that the lower part $S_l$ is not a negligible part of $S$
when
\begin{equation}
\frac{\pi(S_l)}{\pi(S_u)} \ge
\frac{1}{4\sqrt{2}\poly_2(n)}\left(1-\frac{1}{\sqrt{2}}\right). \label{eq:non-negligible}
\end{equation}
\item  We say that  the lower part $S_l$ is a negligible part of $S$ when
\begin{equation}
\frac{\pi(S_l)}{\pi(S_u)} < \frac{1}{4\sqrt{2}\poly_2(n)}
\left(1-\frac{1}{\sqrt{2}}\right). \label{eq:negligible}
\end{equation}
\end{enumerate}
Our plan is the following: the conditional flow $\Psi (S)$ is positive
on any non-empty subset
and it obeys:
$$
\Psi(S) = \Psi'(S_l) \frac{\pi(S_l)}{\pi(S)}+ \Psi'(S_u)\frac{\pi(S_u)}{\pi(S)},
$$
where
$$
\Psi'(S_l):=\frac{1}{\pi(S_l)}\sum_{x\in S_l, y\in \bar{S}} \pi(x) T(y|x)\quad \mbox{ and }\quad \Psi'(S_u) := \frac{1}{\pi(S_u)}\sum\limits_{x\in S_u, y\in \bar{S}} \pi(x) T(y|x).
$$
In other words, $\Psi'(S_l)$ and $\Psi'(S_u)$ are defined as the flow going from $S_l$ and $S_u$ and leaving $S$.

The value $\Psi(S)$ cannot be too small, if at least one of $\Psi'(S_l)$ or $\Psi' (S_u)$ is big enough (and the associated fraction $\pi(S_{l})/\pi(S)$ or $\pi(S_{u})/\pi(S)$). In Case 1 we will show that  $\Psi' (S_l)$ itself is big enough. To that end it will be sufficient to consider the part which leaves $S_l$ but not $Y^l$ (this guarantees that it goes out of $S$, see also Fig.~\ref{fig:case1}). For Case 2
we will consider $\Psi' (S_u),$ particularly that part of it which  goes from $S_u$ to $Y^l \setminus S_l$ (and then going out of $S$, not only $S_u$, see also Fig.~\ref{fig:case2}).

\begin{figure}[h]
\begin{center}
\begin{tikzpicture}
\begin{scope}[>=latex]
\pgfsnakesegmentamplitude=10pt
\draw [snake=brace, very thick] (0,5.2) -- (5.5,5.2);
\node at (2.75, 6) {$\mathbf{Y}^l$};
\draw (.25,4) ellipse (.25 cm and 1cm);
\draw (1.25,3) ellipse (.25 cm and 2cm);
\draw [fill=black] (1.25,4.5) ellipse (.13 cm and .49cm);
\draw [->] (1.25,4) -- (1.25, 2.1);
\draw (2.25,4.5) ellipse (.25 cm and .5cm);
\draw [fill=black] (2.25,4.85) ellipse (.12 cm and .15cm);
\draw [->] (2.25,4.7) -- (2.25, 4.3);
\draw (3.25,3.9) ellipse (.25 cm and 1.1cm);
\draw (4.25,3.6) ellipse (.25 cm and 1.4cm);
\draw [fill=black] (4.25,4.2) ellipse (.15 cm and .75cm);
\draw [->] (4.25,3.5) -- (4.25, 2.3);
\draw (5.25,3.2) ellipse (.25 cm and 1.8cm);
\draw [fill=black] (5.25,4.5) ellipse (.15 cm and .5cm);
\draw [->] (5.25,4) -- (5.25, 3);
\draw [snake=brace, very thick] (6,5.2) -- (11.5,5.2);
\node at (8.75, 6) {$\mathbf{Y}^u$};
\draw (6.25,4) ellipse (.25 cm and 1cm);
\draw [fill=black] (6.25,4.2) ellipse (.15 cm and .8cm);
\draw (7.25,3.7) ellipse (.25 cm and 1.3cm);
\draw [fill=black] (7.25,3.85) ellipse (.2 cm and 1.15cm);
\draw (8.25,3.2) ellipse (.25 cm and 1.8cm);
\draw [fill=black] (8.25,3.5) ellipse (.2 cm and 1.5cm);
\draw (9.25,4.52) ellipse (.25 cm and .5cm);
\draw [fill=black] (9.25,4.7) ellipse (.2 cm and .3cm);
\draw (10.25,2.9) ellipse (.25 cm and 2.1cm);
\draw [fill=black] (10.25,3.1) ellipse (.2 cm and 1.9cm);
\draw (11.25,3.8) ellipse (.25 cm and 1.2cm);
\draw [fill=black] (11.25,3.9) ellipse (.2 cm and 1.1cm);
\end{scope}
\end{tikzpicture}
\end{center}
\caption{When $S_l$ is not a negligible part of $S$, there is a
  considerable flow going out from $S_l$ to within $Y^l$, implying
   that the conditional flow going out from $S$ cannot be small. See text for
  details and rigorous calculations.}\label{fig:case1}
\end{figure}
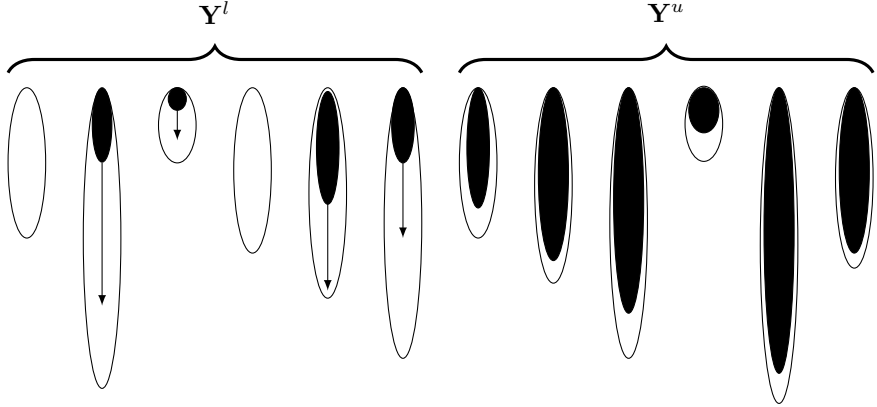

\medskip\noindent
In {\bf Case} 1, the flow going out from $S_l$ within $Y^l$ is sufficient to prove
that the conditional flow going out from $S$ is not negligible.
We know that for any particular $x$, we have a
rapidly mixing Markov chain $M_x$ over the second coordinate $y$. Let their
smallest conductance be denoted by $\Phi_X$. Since all these
Markov chains are rapidly mixing, we have that
$$
\max_{x} \lambda_{M_x,2} \le 1- \frac{1}{\poly_1(n)}
$$
or, equivalently:
$$
\Phi_X \ge \frac{1}{2\poly_1(n)}.
$$
However, in the lower part, for any particular $x$ one has:
$$
\pi_{Y_x}(Y_x(S)) = \frac{\pi(Y_x(S))}{\pi(Y_x)} \le \frac{1}{\sqrt{2}}
$$
so for any fixed $x$ belonging to $L$ it holds that
\begin{eqnarray*}
\frac{1}{2\poly_1(n)}\min\left\{\pi_{Y_x}(Y_x(S)),\left(1-\frac{1}{\sqrt{2}}\right)\right\} \le &&
 \\
\le \sum_{(x,y) \in S, (x,y') \in \bar{S}} \pi_{Y_x}((x,y)) T((x,y')|(x,y)) &&
\end{eqnarray*}
using the modified Cheeger inequality (Lemma~\ref{lemma:mod}). Observing that
$$
\pi_{Y_x}((x,y)) = \frac{\pi((x,y))}{\pi(Y_x)}\;,
$$
we obtain:
\begin{eqnarray*}
\frac{1}{2\poly_1(n)}\pi(Y_x(S))\left(1-\frac{1}{\sqrt{2}}\right)
&\le & \frac{1}{2\poly_1(n)}\min\left\{\pi(Y_x(S)),\pi(Y_x)\left(1-\frac{1}{\sqrt{2}}\right)\right\} \le\\
&\le& \sum_{(x,y) \in S, (x,y') \in \bar{S}} \pi((x,y)) T((x,y')|(x,y))\;.
\end{eqnarray*}
Summing this for all the $x$s belonging to $L$, we deduce that
\begin{eqnarray*}
\pi(S_l) \frac{1}{2\poly_1(n)}\left(1-\frac{1}{\sqrt{2}}\right) \le \\
\sum_{x|Y_x(S) \subseteq S_l}\left( \sum_{(x,y) \in S, (x,y') \in \bar{S}} \pi((x,y))
T((x,y')|(x,y)) \right ).\label{eq:justbefore}
\end{eqnarray*}
Note that the flow on the right hand side of Equation~\ref{eq:justbefore} is not only going out from $S_l$ but also from the
entire $S$. Therefore, we have that
$$
\Psi(S) \ge \frac{\pi(S_l)}{\pi(S)} \times \frac{1}{2\poly_1(n)}\left(1-\frac{1}{\sqrt{2}}\right).
$$
Either $\pi(S_l) \le \pi(S_u)$, which then yields
$$\frac{\pi(S_l)}{\pi(S)} = \frac{\pi(S_l)}{\pi(S_l) + \pi(S_u)} \ge
\frac{\pi(S_l)}{2\pi(S_u)} \ge \frac{1}{8\sqrt{2}\poly_2(n)} \left(1-\frac{1}{\sqrt{2}}\right)  $$
after using Equation~\ref{eq:non-negligible}, or
$\pi(S_l) > \pi(S_u)$, in which case we have
$$\frac{\pi(S_l)}{\pi(S)} > \frac{1}{2} \ge \frac{1}{8\sqrt{2}\poly_2(n)}
\left(1-\frac{1}{\sqrt{2}}\right). $$
(Note that $\poly_2(n) > 1$.) Thus in both cases the following inequality holds:
$$
\Psi(S) \ge \frac{1}{8\sqrt{2}\poly_2(n)}\left(1-\frac{1}{\sqrt{2}}\right) \times
\frac{1}{2\poly_1(n)}\left(1-\frac{1}{\sqrt{2}}\right).
$$

\bigskip\noindent
In {\bf Case} 2, the lower part of $S$ is a negligible part of $S$.
\begin{figure}[h!]
\begin{center}
\begin{tikzpicture}
\begin{scope}[>=latex]
\pgfsnakesegmentamplitude=10pt
\draw [snake=brace, very thick] (0,5.2) -- (5.5,5.2);
\node at (2.75, 6) {$\mathbf{Y}^l$};
\draw (.25,4) ellipse (.25 cm and 1cm);
\draw [->] (10.25,4) -- (.25,4);
\draw (1.25,3) ellipse (.25 cm and 2cm);
\draw [fill=black] (1.25,4.7) ellipse (.10 cm and .29cm);
\draw (2.25,4.5) ellipse (.25 cm and .5cm);
\draw [fill=black] (2.25,4.85) circle (.14 cm);
\draw (3.25,3.9) ellipse (.25 cm and 1.1cm);
\draw [->] (8.25,3) -- (3.25,3);
\draw (4.25,3.6) ellipse (.25 cm and 1.4cm);
\draw [->] (6.25,4.2) -- (4.25,4.2);
\draw [->] (11.25,3.5) -- (4.25,3.5);
\draw (5.25,3.2) ellipse (.25 cm and 1.8cm);
\draw [fill=black] (5.25,4.7) ellipse (.125 cm and .28cm);
\draw [->] (9.25,4.75) -- (5.1,4.1);
\draw [->] (7.25,3.8) -- (5.25,3.8);
\draw [snake=brace, very thick] (6,5.2) -- (11.5,5.2);
\node at (8.75, 6) {$\mathbf{Y}^u$};
\draw (6.25,4) ellipse (.25 cm and 1cm);
\draw [fill=black] (6.25,4.2) ellipse (.15 cm and .8cm);
\draw (7.25,3.7) ellipse (.25 cm and 1.3cm);
\draw [fill=black] (7.25,3.85) ellipse (.2 cm and 1.15cm);
\draw (8.25,3.2) ellipse (.25 cm and 1.8cm);
\draw [fill=black] (8.25,3.5) ellipse (.2 cm and 1.5cm);
\draw (9.25,4.52) ellipse (.25 cm and .5cm);
\draw [fill=black] (9.25,4.7) ellipse (.2 cm and .3cm);
\draw (10.25,2.9) ellipse (.25 cm and 2.1cm);
\draw [fill=black] (10.25,3.1) ellipse (.2 cm and 1.9cm);
\draw (11.25,3.8) ellipse (.25 cm and 1.2cm);
\draw [fill=black] (11.25,3.9) ellipse (.2 cm and 1.1cm);
\end{scope}
\end{tikzpicture}
\end{center}
\caption{When $S_l$ is a negligible part of $S$, there is a   considerable
flow going out from $S_u$ into $Y^l \setminus S_l$. See text for  details
and rigorous calculations.}\label{fig:case2}
\end{figure}
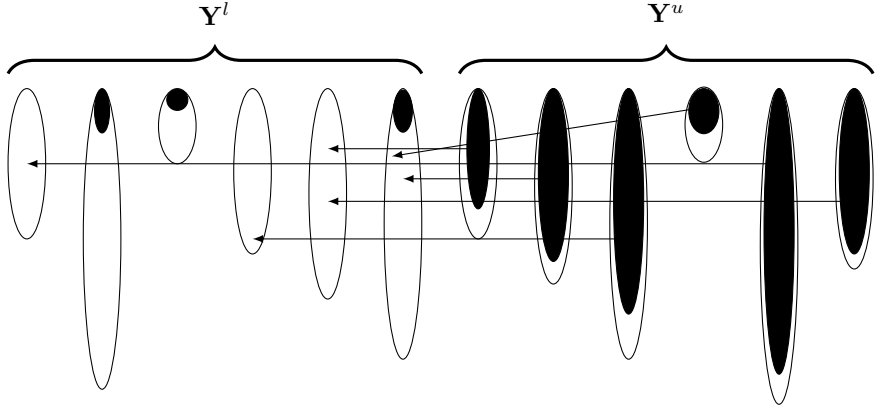
%In this case,
We have that
$$
\pi_X(X(S_u)) \le \frac{1}{\sqrt{2}}
$$
otherwise $\pi(S_u) > 1/2$ would happen (due to the definition of the upper part), and then $\pi(S) > 1/2$, a contradiction.

Hence in the Markov chain $M'$, based on the Lemma~\ref{lemma:mod}, we obtain for $X(S_u)$ that
\begin{equation}\label{eq:uppCheager}
\frac{1}{2\poly_2(n)}\min\left\{\pi_X(X(S_u)),\left(1-\frac{1}{\sqrt{2}}\right)\right\} \le \sum_{{ x'
\in \overline{X(S_u)} \atop x \in X(S_u)}} \pi_X(x) T'(x'|x).
\end{equation}
For all $y$ for which $(x,y) \in S_u$, due to Equation~(\ref{eq:condition}), we can write:
$$
T'(x'|x) \le \sum_{y'} T((x',y')|(x,y))\;.
$$
Multiplying this with $\pi((x,y))$ then summing for all suitable $y$:
$$
\pi(Y_x(S)) T'(x'|x) \le \sum_{y|(x,y) \in S_u} \sum_{y'} \pi((x,y)) T((x',y')|(x,y))
$$
(note that $x \in U$ and thus $Y_x(S) = Y_x(S_u)$) and thus
$$
T'(x'|x) \le \frac{\sum_{y|(x,y) \in S_u} \sum_{y'} \pi((x,y)) T((x',y')|(x,y))}{\pi(Y_x(S)) }.
$$
Inserting this into Equation~\ref{eq:uppCheager}, we find that
\begin{eqnarray*}
\frac{1}{2\poly_2(n)}\min\left\{\pi_X(X(S_u)),\left(1-\frac{1}{\sqrt{2}}\right)\right\}
\le   \\
\le \sum_{x \in X(S_u), x' \in \overline{X(S_u)}} \frac{\pi_X(x)}{\pi(Y_x(S))}
\sum_{y|(x,y) \in S_u} \sum_{y'} \pi((x,y)) T((x',y')|(x,y)).
\end{eqnarray*}
Recall, that $\pi_X(x) = \pi(Y_x)$, and thus $\frac{\pi_X(x)}{\pi(Y_x(S))} \le \sqrt{2}$ for all $x \in X(S_u)$. Therefore we can write that
\begin{eqnarray*}
\frac{1}{2\poly_2(n)}\min\left\{\pi_X(X(S_u)),\left(1-\frac{1}{\sqrt{2}}\right)\right\} \le  \\
\sqrt{2}\sum_{(x,y) \in S_u}\left(  \sum_{(x',y')|x' \in \overline{X(S_u)}} \pi((x,y)) T((x',y')|(x,y)) \right).
\end{eqnarray*}
Note that $\pi(S_u) \le \pi_X(X(S_u)) <1$, and since both items in the minimum taken in the LHS are smaller than 1, their product will be smaller than any of them. Therefore we have
\begin{eqnarray*}
&& \frac{1}{2\sqrt{2}\poly_2(n)} \pi(S_u) \left(1-\frac{1}{\sqrt{2}}\right) \le \\
&& \le \sum_{(x,y) \in S_u}\left( \sum_{(x',y')|x' \in \overline{X(S_u)}} \pi((x,y)) T((x',y')|(x,y)) \right ).
\end{eqnarray*}
This flow is going out from $S_u$, and it is so large that at most half of it can be  picked up by the lower part of $S$ (due to reversibility and due to Equation~\ref{eq:negligible}), and thus the remaining part, i.e., at least half of the flow must go out of $S$. Therefore:
$$
\frac{\pi(S_u)}{\pi(S)} \times \frac{1}{4\sqrt{2}\poly_2(n)} \left(1-\frac{1}{\sqrt{2}}\right) \le \Psi(S) \;.
$$
However, since $S_u$ dominates $S$, namely, $\pi(S_u) > \frac{\pi(S)}{2}$,
we have that
$$
\frac{1}{8\sqrt{2}\poly_2(n)}\left(1-\frac{1}{\sqrt{2}}\right) \le \Psi(S).
$$
Comparing the bounds from Case 1 and Case 2, for all $S$ satisfying $0 < \pi(S) \le \frac{1}{2}$, we can write:
$$
\frac{1}{16\sqrt{2}\poly_2(n) \poly_1(n)}\left(1-\frac{1}{\sqrt{2}}\right)^2 \le \Psi(S).
$$
And thus, for the conductance of the Markov chain $M$ (which is the minimum over all possible $S$)
$$
\frac{1}{16\sqrt{2}\poly_2(n) \poly_1(n)}\left(1-\frac{1}{\sqrt{2}}\right)^2 \le \Phi_M.
$$
Applying this to the Cheeger inequality, one obtains
$$
\lambda_{M,2} \le 1- \frac{\left(\frac{1}{16\sqrt{2}\poly_2(n)  \poly_1(n)}\left(1-\frac{1}{\sqrt{2}}\right)^2\right)^2}{2}
$$
and thus
$$
\frac{1}{1-\lambda_{M,2}} \le \frac{256 \poly_1^2(n) \poly_2^2(n)}{\left(1-\frac{1}{\sqrt{2}}\right)^4}
$$
which is what we wanted to prove.
\qquad\end{proof}

Martin and Randall \cite{MR06} have developed a similar theorem. They assume a disjoint decomposition of the state space $\Omega$ of an irreducible and reversible Markov chain defined via the transition probabilities $P(y|x)$. They require that the Markov chain be rapidly mixing when restricted onto each partition $\Omega_i$ ($\Omega = \cup_i \Omega_i$) and furthermore, another Markov chain, the so-called projection Markov chain $\overline{P}(i|j)$ defined over the indices of the partitions be also rapidly mixing. If all these hold, then the original Markov chain is also rapidly mixing. For the projection Markov chain they use the normalized conditional flow
\begin{equation}
\overline{P}(j|i) = \frac{1}{\pi(\Omega_i)} \sum_{x\in\Omega_i, y\in \Omega_j} \pi(x) P(y|x) \label{MR}
\end{equation}
as transition probabilities. This can be interpreted as a weighted average transition
probability between two partitions, while in our case, Equation~(\ref{eq:condition}) requires only that the transition probability of the lower bounding Markov chain is not more than the minimum of the sum of the transition probabilities going out from one member of the partition (subset $Y_{x_1}$) to the other member of the partition (subset $Y_{x_2}$) with the minimum taken over all the elements of $Y_{x_1}$. Obviously, it is a stronger condition that our Markov chain must be rapidly mixing, since a Markov chain is mixing slower when each transition probability between any two states is smaller. (The latter statement is based on a comparison theorem by Diaconis and Saloff-Coste \cite{DSC93}.)  Therefore, from that point of view, our theorem is weaker. On the other hand, the average transition probability (Equation~(\ref{MR})) is usually hard to calculate, and in this sense our theorem is more applicable. Note that Martin and Randall have also resorted in the end to using chain comparison techniques (Sections 2.2 and 3 in their paper) employing a Metropolis-Hastings chain as a lower bounding chain instead of the projection chain above. Our theorem, however, provides a direct proof of a similar statement.

\section{The RSO Markov chain on balanced realizations}\label{sec:RSO}
In this section we will apply Theorem~\ref{theo:general} to prove that the RSO Markov chain is rapidly mixing on the balanced JDM realizations. We partition its state space according to the vectors $x$ of the auxiliary graph collections (see Definition \ref{def:union} and its explanations). The following result will be used to prove that all derived (marginal) Markov chains $M_x$ are rapidly mixing.
\begin{theorem}
\label{theo:directproduct}
Let ${\cal M}$ be a class of Markov chains whose state space is a $K$ dimensional direct product of spaces, and the problem size of a particular chain is denoted by $n$ $($where we assume that $K=O(poly_1(n))).$

Any transition of the Markov chain $M \in {\cal M}$ changes only one coordinate (each coordinate with equal probability), and the transition probabilities do not depend on the other coordinates. The transitions on each coordinate form irreducible, aperiodic Markov chains (denoted by $M_1, M_2, \ldots M_K$), which are reversible with respect to a distribution $\pi_i.$ Furthermore, each of $M_1, \ldots M_K$ are  rapidly mixing, i.e., with the relaxation time $\frac{1}{1-\lambda_{2,i}}$ being bounded by a $O(\mathrm{poly}_2 (n))$ for all $i$. Then the Markov chain $M$ converges rapidly to the direct product of the $\pi_i$ distributions, and the second largest eigenvalue of $M$ is
$$
\lambda_{2,M} = \frac{K-1+\max_{i}\left\{ \lambda_{2,i}\right\}}{K}
$$
and thus the relaxation time of $M$ is also polynomially bounded:
$$
\frac{1}{1- \lambda_{2,M}} = O(\mathrm{poly}_1(n) \mathrm{poly}_2(n)).
$$
\end{theorem}

\begin{proof}
The transition matrix of $M$ can be described as
$$
\frac{\sum_{i=1}^{K} \left [ \bigotimes_{j = 1}^{i-1}  \mathbf{I}_j
  \right ] \otimes \mathbf{M}_i \otimes \bigotimes_{j=i+1}^K \mathbf{I}_j}{K}
$$
where $\otimes$ denotes usual tensor product from linear  algebra, $\mathbf{M}_i$ denotes the the transition matrix of the Markov chain on the $i$th coordinate, $\mathbf{I}_j$ denotes the identical matrix with the same size as $\mathbf{M}_j$. Since all pairs of terms in the sum above commute, the eigenvalues of $M$ are
$$
\left\{ \frac{1}{K} \sum_{i=1}^K \lambda_{j_i,i} : 1 \le j_i \le |\Omega_i|\right\}
$$
where $\Omega_i$ is the state space of the Markov chain $M_i$ on the $i$th coordinate. The
second largest eigenvalue of $M$ is then obtained from combining the maximal second largest eigenvalue  (maximal among all the second largest eigenvalues of the component transition matrices) with the other largest eigenvalues, i.e., with all others being $1$s:
$$
\frac{K-1+\max_{i}\left\{ \lambda_{2,i}\right\}}{K}\;.
$$
If $g$ denotes the smallest spectral gap, ie. $ g=1-\max_{i}\left\{  \lambda_{2,i} \right\}$, then from above the second largest eigenvalue of $M$ is
$$
\frac{K-g}{K}=1-\frac{g}{K}
$$
namely, the second largest eigenvalue of $M$ is only $K$ times closer to $1$ than the maximal second largest eigenvalue of the individual Markov chains.
\qquad\end{proof}

Next, we announce two theorems that are direct extensions of statements for fast mixing swap Markov chains for regular degree sequences (Cooper, Dyer and Greenhill \cite{cdg})  and for half-regular bipartite degree sequences (Erd\H{o}s, Kiss, Mikl\'os and Soukup \cite{ekms}).
\begin{theorem}
\label{theo:regular}
The swap Markov chain on the realizations of almost regular degree sequences is rapidly mixing.
\end{theorem}
\begin{theorem}
\label{theo:half-regular}
The swap Markov chain on the realizations of almost half-regular bipartite degree sequences is rapidly mixing.
\end{theorem}

The proofs of these results are not directly relevant to this paper. They are based on slight extensions of similar theorems presented in papers \cite{cdg} and \cite{ekms} with very long proofs. Here we only provide brief sketches for the proofs, in the Appendix.

\medskip\noindent
We are now ready to prove the main theorem.
\begin{proof} (Theorem~\ref{theo:main})
We show that the RSO Markov chain on balanced realizations fulfills the conditions in
Theorem~\ref{theo:general}. First we show that condition $(i)$ of Theorem~\ref{theo:general} holds. When restricted to the partition $Y_x$ (that is with $x$ fixed), the RSO Markov chain over the balanced realizations walks on the union of almost semi-regular and almost regular graphs. By restriction here we mean that all
probabilities which would (in the original chain) leave $Y_x$ are put onto the shelf-loop probabilities. Since an RSO changes only one coordinate at a time, independently of other coordinates, all the conditions in Theorem~\ref{theo:directproduct} are fulfilled. Thus the relaxation time of the RSO Markov chain restricted onto $Y_x$ is bounded from above by the relaxation time of the chain restricted onto that coordinate (either an almost semi-regular bipartite or an almost regular graph) on which this restricted chain is the slowest (the smallest
gap). However, based on Theorems~\ref{theo:regular}~and~\ref{theo:half-regular}, all these restrictions are fast mixing, and thus by  Theorem~\ref{theo:directproduct} the polynomial bound in $(i)$ holds. (Here $K = \frac{k(k+1)}{2}$, see Definition~\ref{def:union} and note that an almost semi-regular bipartite graph is also an almost half-regular bipartite graph.)

Next we show that condition $(ii)$ of Theorem~\ref{theo:general} also holds. The first coordinate is the union of auxiliary bipartite graphs, all of which are half-regular. The $M'$ Markov chain corresponding to Theorem~\ref{theo:general} is the swap Markov chain on these auxiliary graphs. Here each possible swap has a probability
$$
\frac{1}{n(n-1)(n-2)(n-3)}
$$
and by Theorem~\ref{theo:swapexist} it is guaranteed that condition \ref{eq:condition} is fulfilled.  Since, again all conditions of Theorem~\ref{theo:directproduct} are fulfilled (mixing is fast within any coordinate due to Theorems~\ref{theo:regular}~and~\ref{theo:half-regular}), the $M'$ Markov chain is also fast mixing.  Condition in Equation~(\ref{eq:condition}) holds due to Theorem~\ref{theo:swapexist}. Since all conditions in Theorem~\ref{theo:general} hold, the RSO swap Markov chain on balanced realizations is also rapidly mixing.
\qquad\end{proof}

\section{Conclusions}\label{sec:con}

We have introduced a swap Markov chain over the space of balanced realizations of an arbitrary JDM, and therefore of {\em arbitrary degree sequences}, and proved that it is fast mixing.  Our proof  is based on the following observations and intermediate results. Any balanced realization can be represented as the labeled union of almost regular and almost semi-regular bipartite graphs. Every balanced realization induces a collection of auxiliary bipartite graphs that are all half-regular and which can be naturally used to generate a disjoint partition of the state space of all balanced realizations. Using conductance methods we then directly proved a general theorem for fast mixing of Markov chains over such structured state spaces, which is similar to an earlier result by Martin and Randall \cite{MR06}. We have also provided extensions to the existing proofs for MCMC fast mixing in the spaces of almost regular graphs based on results of Cooper, Dyer and Greenhill, \cite{cdg}, and of almost half-regular bipartite graphs based on results of Erd\H{o}s, Kiss, Mikl\'os and Soukup  \cite{ekms} and on results of Mikl\'os, Erd\H{o}s and Soukup \cite{mes}.

The obvious open question is the existence of a fast mixing Markov chain for sampling from the full space of simple graphs realizing a given JDM. Since a given JDM also uniquely determines the degree sequence, this could provide an important insight towards proving fast mixing for the degree based MCMC problem, which currently is still open.

\section*{Acknowledgements} B\'alint T\'oth is thanked for his comments related to
Theorem \ref{theo:directproduct}. The authors express their thanks to K.E. Bassler and the Max-Planck-Institut f\"{u}r Physik Komplexer Systeme  for its kind hospitality within its ``Adaptive Networks" advanced study group program, where the  writing of this paper has been completed.

\Appendix
\section{Sketches of proofs for Theorems \ref{theo:regular} and \ref{theo:half-regular} }
\label{sec:app}
The proofs below are direct continuations of the proofs in papers \cite{cdg} and \cite{ekms}. They can be followed within the language and the context of those two papers, which, however, we do not reproduce here, for reasons of brevity.

\medskip
\noindent {\bf Sketch the proof of Theorem \ref{theo:regular} }:
The proof is based on \cite{cdg}. In that paper, the authors,  Cooper, Dyer and  Greenhill, prove the rapid mixing nature of the swap Markov chain in case of regular graphs. The only lemma where they use regularity is Lemma 3, which claims the following: consider a graph in which at most $4$ edges are ``bad'', meaning that they have an assigned value $-1$ or $2$, and they form a subgraph of one of the following $5$ configurations shown on Figure~\ref{fig:5cases}. All other edges get a value $1$, and for each vertex, the sum of the assigned values of its edges is a constant $d \le n/2$, where $n$ is the number of vertices of the graph. Then at most $3$ switches are sufficient to transform this graph into a graph that does not contain any bad edges. A switch operates on $4$ vertices $v_1,v_2,v_3,v_4$, and increases by $1$ the assigned value of edges $(v_1,v_2)$ and $(v_3,v_4)$ (if the edge is not present, then an edge is added with
value $1$, if the assigned value was $-1$, the edge is deleted) and decreases by $1$ the assigned values of edges $(v_2,v_3)$ and $(v_4,v_1)$ (if the modified value is $0$, then the edge is deleted).
\begin{figure}
\begin{center}
\begin{tikzpicture}[auto,scale=.32]
%\draw node [circle,fill=black] at (0,0) (p1) {};
\draw [fill=black] (0,0) circle (10pt);
\draw [fill=black] (3,0) circle (10pt);
\draw [fill=black] (3,-3) circle (10pt);
\draw [fill=black] (3,3) circle (10pt);
\draw [fill=black] (7,-2) circle (10pt);
\draw [fill=black] (7,2) circle (10pt);
\draw (0,0) -- ( 3,0) ; \draw node at (1.4,0.8) {{{\bf 2}}};
\draw (3,-3) --(3,0); \draw node at (3.8,1.5) {{ {\bf -1}}};
\draw (3,3) --(3,0); \draw node at (3.8,-1.5) {{ {\bf -1}}};
\draw (7,-2) --(7,2); \draw node at (7.8,0) {{{\bf ?}}};
\draw node at (4.8,-4) {{{}}};
\end{tikzpicture}
\qquad \raisebox{1cm}{{\Large {\bf OR}}}\qquad
\begin{tikzpicture}[auto,scale=.32]
\draw [fill=black] (0,0) circle (10pt);
\draw [fill=black] (3,0) circle (10pt);
\draw [fill=black] (3,-3) circle (10pt);
\draw [fill=black] (3,3) circle (10pt);
\draw [fill=black] (6,-3) circle (10pt);
\draw (0,0) -- ( 3,0) ; \draw node at (1.4,0.8) {{{\bf 2}}};
\draw (3,-3) --(3,0); \draw node at (3.8,1.5) {{ {\bf -1}}};
\draw (3,3) --(3,0); \draw node at (3.8,-1.5) {{ {\bf -1}}};
\draw (3,-3) --(6,-3); \draw node at (4.8,-4) {{{\bf ?}}};
\end{tikzpicture}
\qquad \raisebox{1cm}{{\Large {\bf OR}}}\qquad
\begin{tikzpicture}[auto,scale=.32]
\draw [fill=black] (0,0) circle (10pt);
\draw [fill=black] (3,0) circle (10pt);
\draw [fill=black] (3,-3) circle (10pt);
\draw [fill=black] (3,3) circle (10pt);
\draw [fill=black] (6,-3) circle (10pt);
\draw (0,0) -- ( 3,0) ; \draw node at (1.4,0.8) {{{\bf -1}}};
\draw (3,-3) --(3,0); \draw node at (3.8,1.5) {{ {\bf -1}}};
\draw (3,3) --(3,0); \draw node at (3.8,-1.5) {{ {\bf 2}}};
\draw (3,-3) --(6,-3); \draw node at (4.8,-4) {{{\bf ?}}};
\end{tikzpicture}

\qquad \raisebox{1cm}{{\Large {\bf OR}}}\qquad
\begin{tikzpicture}[auto,scale=.32]
\draw [fill=black] (0,0) circle (10pt);
\draw [fill=black] (3,0) circle (10pt);
\draw [fill=black] (3,-3) circle (10pt);
\draw [fill=black] (3,3) circle (10pt);
\draw (0,0) -- ( 3,0) ; \draw node at (1.4,0.8) {{{\bf 2}}};
\draw (3,-3) --(3,0); \draw node at (3.8,1.5) {{ {\bf -1}}};
\draw (3,3) --(3,0); \draw node at (3.8,-1.5) {{ {\bf -1}}};
\draw (0,0) -- (3,-3);  \draw node at (1,-2) {{{\bf ?}}};

\end{tikzpicture}
\qquad \raisebox{1cm}{{\Large {\bf OR}}}\qquad
\begin{tikzpicture}[auto,scale=.32]
\draw [fill=black] (0,0) circle (10pt);
\draw [fill=black] (3,0) circle (10pt);
\draw [fill=black] (3,-3) circle (10pt);
\draw [fill=black] (3,3) circle (10pt);
\draw (0,0) -- ( 3,0) ; \draw node at (1.4,0.8) {{{\bf -1}}};
\draw (3,-3) --(3,0); \draw node at (3.8,1.5) {{ {\bf 2}}};
\draw (3,3) --(3,0); \draw node at (3.8,-1.5) {{ {\bf -1}}};
\draw (0,0) -- (3,-3);  \draw node at (1,-2) {{{\bf ?}}};
\end{tikzpicture}
\end{center}
\caption{$5$ possible ``bad'' configurations. The edge with a ``?''
  might get a value of either $-1$ or $2$. See text for details.}\label{fig:5cases}
\end{figure}
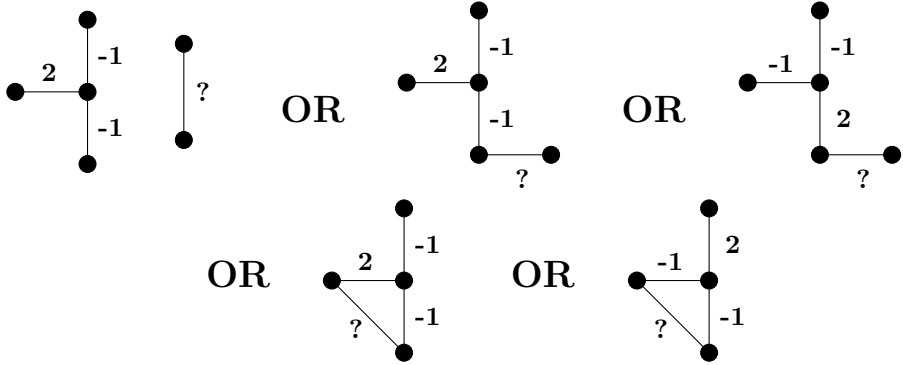

We prove that a similar lemma holds in the situation when the degrees are both $d \ge
0$ and $d+1 \le n$, in which case at most $4$ switches are necessary. The $4$ switches instead of $3$ causes a higher order, but still polynomial upper bound on the relaxation time.

The first observation is that the graphs with bad edges are obtained by a matrix $G+G'-Z$, where all $G$, $G'$ and $Z$ are adjacency matrices of graphs with the same degree sequence. What follows is that whenever a vertex has a bad valued edge, the degree of the edge can be neither $0$ nor $n$.

If there is vertex $v_1$ which has both a $2$ and a $-1$ edge, then let denote $v_2$ and $v_3$ the corresponding neighbor vertices, respectively. There is a vertex $v_4$ such that $v_2$ is not connected to $v_4$, and $v_3$ might or might not be connected to $v_4$. We apply a switch on $v_1$, $v_2$, $v_3$ and $v_4$, which removes two bad
edges and creates at most $1$ new bad edge, thus decreasing the number of bad edges by at least $1$.

If there is no vertex incident with different types of bad edges, then consider any bad edge $(v_1,v_2)$. If the assigned value is $-1$, then we have to find a $v_3$ which is connected to $v_2$. If there is a $-1$ valued edge $(v_3,v_4)$, then we apply a switch on $v_1$, $v_2$, $v_3$ and $v_4$, which removes two bad edges and creates at most $1$ new bad edge, thus decreasing the number of bad edges by at least $1$. Otherwise, there must be a vertex $v_4$ which is not connected to $v_3$ but connected to $v_1$, since $d(v_1) \ge d(v_3)-1$, and the difference on the sum of bad values for $v_1$ and $v_3$ is at least $2$. We apply a switch on $v_1$, $v_2$, $v_3$ and $v_4$, which removes the $1$ bad edge, $(v_1,v_2)$.

Finally, if the assigned bad value to edge $(v_1,v_2)$ is $2$, then there must be a vertex $v_3$ such that $v_2$ is not connected to $v_3$. If there is a $2$ valued edge $(v_3,v_4)$, then we apply a switch on $v_1$, $v_2$, $v_3$ and $v_4$, which removes two bad edges and creates at most $1$ new bad edge, thus decreasing the number of bad edges by at least $1$. Otherwise, there must be a vertex $v_4$ which is connected to $v_3$ but not connected to $v_1$, since $d(v_1) -1 \le d(v_3)$, and the difference on the sum of bad values for $v_1$ and $v_3$ is at least $2$. We apply a switch on $v_1$, $v_2$, $v_3$ and $v_4$, which removes the $1$ bad edge $(v_1,v_2)$.

Hence, while there are bad value edges, we can apply a switch that decreases the number of bad edges at least by $1$.  Since there are at most $4$ bad edges, the number of necessary switches is at most $4$.

\medskip\noindent {\bf Sketch the proof of Theorem \ref{theo:half-regular} }:
The proof is based on \cite{ekms}. In that paper, the authors prove the rapid mixing nature of a Markov chain on half-regular degree sequence realizations with a forbidden (possibly empty) star and (also possibly empty) one factor. The only place where they use half-regularity is their Lemma 4.6. In that lemma, the authors prove that a certain 0-1 matrix with at most $3$ possible ``bad'' values, at most two $-1$ values and at most one $2$ value in the same column can be transformed into a $0-1$ matrix using at most $3$ switches. Here we prove if there is no forbidden sub-graph, and the degree sequence is almost half-regular (instead of half-regular), then the same lemma
holds.

Indeed, in that case, the difference in the row sums can be also at most $1$. Consider the row $i$ containing a bad value $2$ in column $j$. There must be a row $l$ containing $0$ in column $j$. The difference between $2$ and $0$ is $2$, while the difference between the sums of rows $i$ and $l$ can be at most $1$, therefore, we have to find another column $m$, where the corresponding values are $0$ and $1$, and a switch on these $4$ values eliminates the bad value $2$ without creating a new bad value. Similar reasoning holds for the bad value $-1$.

\goodbreak
\bibliographystyle{plain}

\end{document}